\documentclass[envcountsame]{llncs}

\usepackage[utf8]{inputenc}
\usepackage{microtype}
\usepackage{cite}
\usepackage{tasks}
\usepackage{amsmath}
\usepackage{graphicx}
\usepackage{mathtools}

\title{First-Hitting Times Under Additive Drift}

\author{Timo Kötzing \and Martin~S.\ Krejca}
\institute{Hasso Plattner Institute, University of Potsdam, Potsdam, Germany\vspace*{-1.5\baselineskip}
}
\authorrunning{T.~Kötzing and M.~S.~Krejca}

\allowdisplaybreaks

\providecommand{\ignore}[1]{}

\usepackage{enumitem}
\usepackage{cleveref}
\usepackage{upgreek}
\usepackage{xparse}
\usepackage{xspace}

\makeatletter
    \def\NAT@spacechar{~}
\makeatother

\let\originalleft\left
\let\originalright\right
\renewcommand{\left}{\mathopen{}\mathclose\bgroup\originalleft}
\renewcommand{\right}{\aftergroup\egroup\originalright}

\makeatletter
\def\IfEmptyTF#1{%
	\if\relax\detokenize{#1}\relax
	\expandafter\@firstoftwo
	\else
	\expandafter\@secondoftwo
	\fi}
\makeatother

\DeclareDocumentCommand{\mathOrText}{m}
{%
    \ensuremath{#1}\xspace%
}

\usepackage{dsfont}
\renewcommand*{\Re}{\mathOrText{\mathds{R}}}
\newcommand*{\Na}{\mathOrText{\mathds{N}}}

\DeclareDocumentCommand{\functionTemplate}{m m m m o}
{%
	\IfNoValueTF{#5}%
	{%
		\mathOrText{#1\left#2{#4}\right#3}%
	}%
	{%
		\mathOrText{#1#5#2{#4}#5#3}%
	}%
}
\DeclareDocumentCommand{\probabilistcFunction}{m m O{} o}
{%
	\functionTemplate{#1}{[}{]}{#2\IfEmptyTF{#3}{}{\,\IfNoValueTF{#4}{\left}{#4}\vert\,\vphantom{#2}{#3}\IfNoValueTF{#4}{\right.}{}}}[#4]%
}

\DeclareDocumentCommand{\Pr}{m O{} o}
{%
	\probabilistcFunction{\mathrm{Pr}}{#1}[#2][#3]%
}
\DeclareDocumentCommand{\E}{m O{} o}
{%
	\probabilistcFunction{\mathrm{E}}{#1}[#2][#3]%
}
\DeclareDocumentCommand{\Var}{m O{} o}
{%
	\probabilistcFunction{\mathrm{Var}}{#1}[#2][#3]%
}

\DeclareDocumentCommand{\bigO}{m o}
{%
	\functionTemplate{\mathrm{O}}{(}{)}{#1}[#2]%
}
\DeclareDocumentCommand{\smallO}{m o}
{%
	\functionTemplate{\mathrm{o}}{(}{)}{#1}[#2]%
}
\DeclareDocumentCommand{\bigTheta}{m o}
{%
	\functionTemplate{\Theta}{(}{)}{#1}[#2]%
}
\DeclareDocumentCommand{\bigOmega}{m o}
{%
	\functionTemplate{\Omega}{(}{)}{#1}[#2]%
}
\DeclareDocumentCommand{\smallOmega}{m o}
{%
	\functionTemplate{\upomega}{(}{)}{#1}[#2]%
}
\DeclareDocumentCommand{\eulerE}{o}
{%
    \mathOrText{\mathrm{e}\IfNoValueTF{#1}{}{^{#1}}}%
}

\DeclareDocumentCommand{\poly}{m o}
{%
	\functionTemplate{\mathrm{poly}}{(}{)}{#1}[#2]%
}

\DeclareDocumentCommand{\range}{m o}
{%
	\functionTemplate{\mathrm{rng}}{(}{)}{#1}[#2]%
}

\graphicspath{../Figures}

\DeclareDocumentCommand{\randomProcess}{o}
{%
    \mathOrText{X\IfNoValueTF{#1}{}{_{#1}}}%
}
\newcommand*{\timePoint}{\mathOrText{t}}
\newcommand*{\firstHittingTime}{\mathOrText{T}}

\DeclareDocumentCommand{\transformedProcess}{o}
{%
    \mathOrText{Y\IfNoValueTF{#1}{}{_{#1}}}%
}
\DeclareDocumentCommand{\otherProcess}{o}
{%
    \mathOrText{Z\IfNoValueTF{#1}{}{_{#1}}}%
}
\DeclareDocumentCommand{\filtration}{o}
{%
    \mathOrText{\mathcal{F}\IfNoValueTF{#1}{}{_{#1}}}%
}

\newcommand*{\specialTimePoint}{\mathOrText{t'}}
\DeclareDocumentCommand{\naturalFiltration}{m}
{%
    \mathOrText{\randomProcess[0], \ldots, \randomProcess[#1]}%
}
\DeclareDocumentCommand{\filtration}{o}
{%
    \mathOrText{\mathcal{F}\IfNoValueTF{#1}{}{_{#1}}}%
}
\newcommand*{\xmin}{\mathOrText{x_{\min}}}

\begin{document}

    \maketitle
    
\begin{abstract}
   For the last ten years, almost every theoretical result concerning the expected run time of a randomized search heuristic used \emph{drift theory}, making it the arguably most important tool in this domain.
   Its success is due to its ease of use and its powerful result: drift theory allows the user to derive bounds on the expected first-hitting time of a random process by bounding expected local changes of the process~-- the \emph{drift}. This is usually far easier than bounding the expected first-hitting time directly.
    
    Due to the widespread use of drift theory, it is of utmost importance to have the best drift theorems possible. We improve the fundamental additive, multiplicative, and variable drift theorems by stating them in a form as general as possible and providing examples of why the restrictions we keep are still necessary. Our additive drift theorem for upper bounds only requires the process to be nonnegative, that is, we remove unnecessary restrictions like a finite, discrete, or bounded search space. As corollaries, the same is true for our upper bounds in the case of variable and multiplicative drift. %
\end{abstract}

\ignore
{
    UTF8 Abstract
    
For the last ten years, almost every theoretical result concerning the expected run time of a randomized search heuristic used drift theory, making it the arguably most important tool in this domain. Its success is due to its ease of use and its powerful result: drift theory allows the user to derive bounds on the expected first-hitting time of a random process by bounding expected local changes of the process – the drift. This is usually far easier than bounding the expected first-hitting time directly.

Due to the widespread use of drift theory, it is of utmost importance to have the best drift theorems possible. We improve the fundamental additive, multiplicative, and variable drift theorems by stating them in a form as general as possible and providing examples of why the restrictions we keep are still necessary. Our additive drift theorem for upper bounds only requires the process to be nonnegative, that is, we remove unnecessary restrictions like a finite, discrete, or bounded search space. As corollaries, the same is true for our upper bounds in the case of variable and multiplicative drift.
    
}

    \section{Drift Theory}
    \label{sec:introduction}

In the theory of randomized algorithms, the first and most important part of algorithm analysis is to compute the expected run time. A finite run time guarantees that the algorithm terminates almost surely, and, due to Markov's inequality, the probability of the run time being far larger than the expected value can be bounded, too. Thus, it is important to have strong and easy to handle tools in order to derive expected run times. The de facto standard for this purpose in the theory of randomized search heuristics is \emph{drift theory}.

Drift theory is a general term for a collection of theorems that consider random processes and bound the expected time it takes the process to reach a certain value~-- the \emph{first-hitting time}. The beauty and appeal of these theorems lie in them usually having few restrictions but yielding strong results. Intuitively speaking, in order to use a drift theorem, one only needs to estimate the expected change of a random process~-- the \emph{drift}~-- at any given point in time. Hence, a drift theorem turns expected local changes of a process into expected first-hitting times. In other words, local information of the process is transformed into global information.

Drift theory gained traction in the theory of randomized search heuristics when it was introduced to the community by He and Yao~\cite{DBLP:journals/ai/HeY01, DBLP:journals/nc/HeY04} via the \emph{additive drift theorem}. However, they were not the first to prove it. The result dates back to Hajek~\cite{hajek1982hitting}, who stated the theorem in a fashion quite different from how it is phrased nowadays. According to Lengler~\cite{DBLP:journals/corr/abs-1712-00964}, the theorem has been proven even prior to that various times. Since then, many different versions of drift theorems have been proven, the most common ones being the \emph{variable drift theorem}~\cite{Johannsen10} and the \emph{multiplicative drift theorem}~\cite{DBLP:conf/gecco/DoerrJW10}. The different names refer to how the drift is bounded other than independent of time: additive means that the drift is bounded by the same value for all states; in a multiplicative scenario, the drift is bounded by a multiple of the current state of the process; and in the setting of variable drift, the drift is bounded by any monotone function with respect to the current state of the process.

At first, the theorems were only stated over finite or discrete search spaces. However, these restrictions are seldom used in the proofs and thus not necessary, as pointed out, for example, by Lehre and Witt~\cite{Lehre2014}, who prove a general drift theorem without these restrictions. Nonetheless, up to date, all drift theorems require a bounded search space;\footnote{Lengler~\cite{DBLP:journals/corr/abs-1712-00964} briefly mentions infinite search spaces and also gives a proof for a restricted version of the additive drift theorem in the setting of an unbounded discrete search space.} Semenov and Terkel~\cite{DBLP:journals/ec/Semenov03} state a Theorem very much like an additive drift theorem for unbounded search spaces, but they require the process to have a bounded variance, as they also prove concentration for their result.

The area of randomized search heuristics is, in fact, in strong need of extended drift theorems and a careful discussion of what happens when restrictions are not met. While most search spaces are finite and, thus, the existing drift theorems sufficient, progress will be inhibited whenever search spaces are not naturally finite. Worse yet, the existing drift theorems might be applied where they are not applicable, as happened when in \cite[Section 4]{doerr2017bounding} the additive drift theorem was applied on an unbounded search space.

While previously new drift theorems were proven on a need-to-have basis using whatever restrictions where present in the concrete application, we aim at providing the best possible theorem for any applications to come. For the restrictions that remain, we give examples that show that these restrictions are, in some sense, necessary. In this way, we want to further the understanding of random processes in general and not just for a concrete application; thus, this work should benefit a lot of future work in the area of randomized search heuristics.

Our most important results are the upper and lower bound of the classical additive drift theorem (Thm.~\ref{thm:addDriftUnbounded} and~\ref{thm:addDriftLowerBoundExpectedBoundedSteps}, respectively), which we prove for unbounded\footnote{For the upper bound, we require the search space to be lower-bounded but not upper-bounded. We still refer to such a setting as unbounded.} search spaces. These theorems are used as a foundation for all of our other drift theorems in other settings.
Overall, our results can be summarized as follows:

For \textbf{additive drift}, we prove an upper bound for any nonnegative process (Thm.~\ref{thm:addDriftUnbounded}), and a lower bound for processes with bounded \emph{expected} step size (Thm.~\ref{thm:addDriftLowerBoundExpectedBoundedSteps}).

For \textbf{multiplicative drift} and \textbf{variable drift}, we prove upper bounds for any nonnegative process (Cor.~\ref{cor:MultiDriftUnbounded} and~\ref{cor:MultiDriftUnboundedNoGap}; and Thm.~\ref{thm:varDriftUnbounded} and~\ref{thm:varDriftUnboundedNoGap}, respectively).

\smallskip
The intention of this paper is to provide a fully-packed reference for very general yet easy-to-apply drift theorems. That is, we try to keep the requirements of the theorems as easy as possible but still state the theorems in the most general way, given the restrictions. Further, we discuss the ideas behind the different theorems and some of the proofs in order to provide insights into how and why drift works, we provide examples, and we discuss prior work at the beginning of each section.

We only consider bounds on the expected first-hitting time, as this is already a vast field to explore. However, we want to mention that drift theory has also brought forth other results than expected first-hitting times, namely, concentration bounds and negative drift, which are related. Both areas bound \emph{the probability} of the first-hitting time taking certain values. Concentration bounds show how unlikely it is for a process to take much longer than the expected first-hitting time~\cite{DBLP:conf/ppsn/DoerrG10a, DBLP:journals/algorithmica/Kotzing16}. On the other hand, negative drift bounds how likely it is for the process to reach the goal although the drift is going the opposite direction \cite{DBLP:journals/algorithmica/Kotzing16, OlivetoW11, OlivetoW12}. These results are also very helpful but out of the scope of this paper.

Our paper is structured as follows: in Section~\ref{sec:prelims}, we start by introducing important notation and terms, which we use throughout the entire paper. Further, we also discuss Theorem~\ref{thm:optionalStopping}, which our proofs of the additive drift theorems rely on. In Section~\ref{sec:additiveDrift}, we discuss additive drift and prove our main results. We then continue with variable drift in Section~\ref{sec:variableDrift} as a generalization of additive drift. In this section, we introduce two different versions of first-hitting time that our results are based on. Last, we consider the scenario of multiplicative drift in Section~\ref{sec:multiplicativeDrift}.

Most of our proofs can be found in the appendix. A shorter version of this paper has been accepted at PPSN~$2018$.

    \section{Preliminaries}
    \label{sec:prelims}
    
We consider the expected \emph{first-hitting time}~\firstHittingTime of a process $(\randomProcess[\timePoint])_{\timePoint \in \Na}$ over~\Re, which we call~\randomProcess[\timePoint] for short. That is, we are interested in the expected time it takes the process to reach a certain value for the first time, which we will refer to as the \emph{target}. Usually, our target is the value~$0$, that is, we will define the random variable $\firstHittingTime = \inf\{\timePoint \mid \randomProcess[\timePoint] \leq 0\}$ (where we define that $\inf\emptyset \coloneqq \infty$).

We provide bounds on~$\E{\firstHittingTime}[\randomProcess[0]]$ with respect to the \emph{drift} of~\randomProcess[\timePoint], which is defined as
\[
    \randomProcess[\timePoint] - \E{\randomProcess[\timePoint + 1]}[\naturalFiltration{\timePoint}]\ .
\]
Note that~$\E{\firstHittingTime}[\randomProcess[0]]$ as well as~$\E{\randomProcess[\timePoint + 1]}[\naturalFiltration{\timePoint}]$ are both random variables. Because of the latter, the drift is a random variable, too. Further note that, if the drift is positive,~\randomProcess[\timePoint] decreases its value in expectation over time when considering positive starting values. This is why~$0$ will be our target most of the time.

We are only interested in the process~\randomProcess[\timePoint] until the time point~\firstHittingTime. That is, all of our requirements only need to hold for all $\timePoint < \firstHittingTime$ (since we also consider $\timePoint + 1$). While this phrasing is intuitive, it is formally inaccurate, as~\firstHittingTime is a random variable. We will continue to use it; however, formally, each of our inequalities in each of our requirements should be multiplied with the characteristic function of the event $\{\timePoint < \firstHittingTime\}$. In this way, the inequalities trivially hold once $\timePoint \geq \firstHittingTime$ and, otherwise, are the inequalities we state. This is similar to conditioning on the event $\{\timePoint < \firstHittingTime\}$ but has the benefit of being valid even if $\Pr{\timePoint < \firstHittingTime} = 0$ holds.

We want to mention that all of our results actually hold for a random process $(\randomProcess[\timePoint])_{\timePoint \in \Na}$ adapted to a filtration~$(\filtration[\timePoint])_{\timePoint \in \Na}$, where~\firstHittingTime is a stopping time defined with respect to~\filtration[\timePoint].\!\footnote{More information on filtrations can be found, for example, in \emph{Randomized Algorithms}~\cite{Motwani:1995:RA:211390} in the section on martingales.} Since this detail is frequently ignored in drift theory, we phrase all of our results with respect to the natural filtration, making them look more familiar to usual drift results. For any time point $\timePoint \leq \firstHittingTime$, we call $\randomProcess[0], \ldots, \randomProcess[\timePoint - 1]$ the \emph{history} of the process.

Last, we state all of our results conditional on~\randomProcess[0], that is, we bound \E{\firstHittingTime}[\randomProcess[0]]. However, by the law of total expectation, one can easily derive a bound for $\E{\firstHittingTime} = \E{\E{\firstHittingTime}[\randomProcess[0]]}[][\big]$.

\subsection{Martingale Theorems}

In this section, we state two theorems that we will use in order to prove our results in the next sections. Both theorems make use of \emph{martingales}, a fundamental concept in the field of probability theory. A martingale is a random process with a drift of~$0$, that is, in expectation, it does not change over time. Further, a \emph{supermartingale} has a drift of at least~$0$, that is, it decreases over time in expectation, and a \emph{submartingale} has a drift of at most~$0$, that is, it increases over time in expectation.

The arguably most important theorem for martingales is the Optional Stopping Theorem (Theorem~\ref{thm:optionalStopping}). We use a version given by Grimmett and Stirzaker~\cite[Chapter~$12.5$, Theorem~$9$]{grimmett2001probability} that can be extended to super- and submartingales.

\begin{theorem}[\textrm{Optional Stopping}]
    \label{thm:optionalStopping}
    Let $(\randomProcess[\timePoint])_{\timePoint \in \Na}$ be a random process over~\Re, and let $\firstHittingTime$ be a stopping time\footnote{Intuitively, for the natural filtration, a stopping time~\firstHittingTime is a random variable over~\Na such that, for all $\timePoint \in \Na$, the event $\{\timePoint \leq \firstHittingTime\}$ is only dependent on $\randomProcess[0], \ldots, \randomProcess[t]$.} for~\randomProcess[\timePoint]. Suppose that
    \begin{enumerate}[label=(\alph*)]
        \item\label{item:optStopFiniteExpectation} $\E{\firstHittingTime} < \infty$ and that
        
        \item there is some value $c \geq 0$ such that, for all $t < \firstHittingTime$, it holds that $\E{|\randomProcess[\timePoint + 1] - \randomProcess[\timePoint]|}[\naturalFiltration{\timePoint}][\big] \leq c$.
    \end{enumerate}
    Then:
    \begin{enumerate}
        \item\label{item:optionalStoppingSupermartingale} If, for all $\timePoint < \firstHittingTime$, $\randomProcess[\timePoint] - \E{\randomProcess[\timePoint + 1]}[\naturalFiltration{\timePoint}] \geq 0$, then $\E{\randomProcess_\firstHittingTime} \leq \E{\randomProcess[0]}$.
        
        \item\label{item:optionalStoppingSubmartingale} If, for all $\timePoint < \firstHittingTime$, $\randomProcess[\timePoint] - \E{\randomProcess[\timePoint + 1]}[\naturalFiltration{\timePoint}] \leq 0$, then $\E{\randomProcess_\firstHittingTime} \geq \E{\randomProcess[0]}$.
    \end{enumerate}
\end{theorem}

Theorem~\ref{thm:optionalStopping} allows us to bound \E{\randomProcess[\firstHittingTime]} independently of its history, which is why our drift results are independent of the history of~\randomProcess[\firstHittingTime] as well.

Note that case~(\ref{item:optionalStoppingSupermartingale}) refers to supermartingales, whereas case~(\ref{item:optionalStoppingSubmartingale}) refers to submartingales. Intuitively, case~(\ref{item:optionalStoppingSupermartingale}) says that a supermartingale will have, in expectation, a lower value than it started with, which makes sense, as a supermartingale decreases over time in expectation. Case~(\ref{item:optionalStoppingSubmartingale}) is analogous for submartingales. For martingales, both cases can be combined in order to yield an equality.

Martingales are essential in the proofs of our theorems. We will frequently transform our process such that it results in a supermartingale or a submartingale in order to apply Theorem~\ref{thm:optionalStopping}.

Another useful theorem for martingales is the following Azuma--Hoeffding Inequality~\cite{azuma1967}. This inequality basically is for martingales what a Chernoff bound is for binomial distributions.

\begin{theorem}[Azuma-Hoeffding Inequality]
    \label{thm:AzumaHoeffding}
    Let $(\randomProcess[\timePoint])_{\timePoint \in \Na}$ be a random process over~\Re. Suppose that
    \begin{enumerate}[label=(\alph*)]
        \item there is some value $c > 0$ such that, for all $\timePoint \in \Na$, it holds that $|\randomProcess[\timePoint] - \randomProcess[\timePoint + 1]| < c$.
    \end{enumerate}
    If, for all $\timePoint \in \Na$, $\randomProcess[\timePoint] - \E{\randomProcess[\timePoint + 1]}[\naturalFiltration{\timePoint}] \geq 0$, then, for all $\timePoint \in \Na$ and all $r > 0$,
    \[
        \Pr{\randomProcess[\timePoint] - \randomProcess[0] \geq r} \leq \eulerE^{-\frac{r^2}{2\timePoint c^2}}\ .
    \]
\end{theorem}

    \section{Additive Drift}
    \label{sec:additiveDrift}
    
We speak of additive drift when the drift can be bounded by a value independent of the process itself. That is, the bound is independent of time and state.

When considering the first-hitting time~\firstHittingTime of a random process $(\randomProcess[\timePoint])_{\timePoint \in \Na}$ whose drift is \emph{lower}-bounded by a value $\delta > 0$, then \E{\firstHittingTime}[\randomProcess[0]] is \emph{upper}-bounded by $\randomProcess[0]/\delta$. Interestingly, if the drift of \randomProcess[\timePoint] is \emph{upper}-bounded by~$\delta$, \E{\firstHittingTime}[\randomProcess[0]] is \emph{lower}-bounded by $\randomProcess[0]/\delta$. Thus, if the drift of \randomProcess[\timePoint] is exactly~$\delta$, that is, we know how much expected progress \randomProcess[\timePoint] makes in each step, our expected first-hitting time is equal to $\randomProcess[0]/\delta$. This result is remarkable, as it can be understood intuitively as follows: since we stop once \randomProcess[\timePoint] reaches~$0$, the distance from our start (\randomProcess[0]) to our goal ($0$) is exactly~\randomProcess[0], and we make an expected progress of~$\delta$ each step. Thus, in expectation, we are done after $\randomProcess[0]/\delta$ steps.

\subsection{Upper Bounds}
\label{subsec:additiveDriftUpperBound}

We give a proof for the Additive Drift Theorem, originally published (in a more restricted version) by He and Yao~\cite{DBLP:journals/ai/HeY01, DBLP:journals/nc/HeY04}. We start by reproving the original theorem (which requires a bounded search space) but in a simpler, more elegant and educational manner. We then greatly extend this result by generalizing it to processes with a bounded step width. Finally, we lift also this restriction.

In all of these cases, we require our random process to only take nonnegative values. The intuitive reason for this is the following: when estimating an upper bound for the expected first-hitting time, we need a lower bound of the drift. This means the larger our bound of the drift, the better our bound for the first-hitting time. Since our process is nonnegative, the drift for values close to~$0$ provides a natural bound for the drift (which is uniform over the entire search space, since we look at \emph{additive} drift). If our process could take values less than~$0$, we could artificially increase our lower bound of the drift for values that are now bounded by~$0$ and, thus, improve our first-hitting time. At the end of this section, we also give an example (Example~\ref{ex:additiveDriftUpperBound}), which shows how our most general drift theorem (Theorem~\ref{thm:addDriftUnbounded}) fails if the process can take negative values.

The proof of the following theorem transforms the process into a supermartingale and then uses Theorem~\ref{thm:optionalStopping}. However, in order to apply Theorem~\ref{thm:optionalStopping}, we have to make sure to fulfill its condition~\ref{item:optStopFiniteExpectation}, which is the hardest part.

\begin{theorem}[\textrm{Upper Additive Drift, Bounded}]
    \label{thm:addDriftBounded}
    Let $(\randomProcess[\timePoint])_{\timePoint \in \Na}$ be a random process over~\Re, and let $\firstHittingTime = \inf\{\timePoint \mid \randomProcess[\timePoint] \leq 0\}$. Furthermore, suppose that,
    \begin{enumerate}[label=(\alph*)]
        \item\label{item:additiveDriftNonNegative} for all $\timePoint \leq \firstHittingTime$, it holds that $\randomProcess[\timePoint] \geq 0$, that
        
        \item\label{item:additiveDriftCondition} there is some value $\delta > 0$ such that, for all $\timePoint < \firstHittingTime$, it holds that $\randomProcess[\timePoint] - \E{\randomProcess[\timePoint + 1]}[\naturalFiltration{\timePoint}] \geq \delta$, and that
        
        \item\label{item:additiveDriftBoundedSpace} there is some value $c \geq 0$ such that, for all $\timePoint < \firstHittingTime$, it holds that $\randomProcess[\timePoint] \leq c$.
    \end{enumerate}
    Then
    \[
        \E{\firstHittingTime}[\randomProcess[0]] \leq \frac{\randomProcess[0]}{\delta}\ .
    \]

\end{theorem}

Note that condition~\ref{item:additiveDriftNonNegative} means that~\firstHittingTime can be rewritten as $\inf\{\timePoint \mid \randomProcess[\timePoint] = 0\}$, that is, we have to hit~$0$ exactly in order to stop. We show in Example~\ref{ex:additiveDriftUpperBound} why this condition is crucial.

Condition~\ref{item:additiveDriftCondition} bounds the expected progress we make each time step. The larger~$\delta$, the lower the expected first-hitting time. However, due to condition~\ref{item:additiveDriftNonNegative}, note that small values of~\randomProcess[\timePoint] create a natural upper bound for~$\delta$, as the progress for such values can be at most $|\randomProcess[\timePoint] - 0| = \randomProcess[\timePoint]$.

Condition~\ref{item:additiveDriftBoundedSpace} means that we are considering random variables over the interval $[0, c]$. It is a restriction that all previous additive drift theorems have but that is actually not necessary, as we show with Theorem~\ref{thm:addDriftUnbounded}. In the following proof, we use this condition in order to show that $\E{\firstHittingTime} < \infty$, which is necessary when applying Theorem~\ref{thm:optionalStopping}.

\begin{proof}[Proof of Theorem~\ref{thm:addDriftBounded}]
    We want to use case~(\ref{item:optionalStoppingSupermartingale}) of the Optional Stopping Theorem in the version of Theorem~\ref{thm:optionalStopping}. Thus, we define, for all $\timePoint < \firstHittingTime$, $\transformedProcess[\timePoint] = \randomProcess[\timePoint] + \delta\timePoint$, which is a supermartingale, since
    \begin{align*}
        \transformedProcess[\timePoint] - \E{\transformedProcess[\timePoint + 1]}[\transformedProcess[0], \ldots, \transformedProcess[\timePoint]]
        &= \randomProcess[\timePoint] + \delta\timePoint - \E{\randomProcess[\timePoint + 1] + \delta(\timePoint + 1)}[\naturalFiltration{\timePoint}]\\
        &= \randomProcess[\timePoint] - \E{\randomProcess[\timePoint + 1]}[\naturalFiltration{\timePoint}] - \delta \geq 0\ ,
    \end{align*}
    as we assume that $\randomProcess[\timePoint] - \E{\randomProcess[\timePoint + 1]}[\naturalFiltration{\timePoint}] \geq \delta$ for all $\timePoint < \firstHittingTime$. Note that we can change the condition $\transformedProcess[0], \ldots, \transformedProcess[\timePoint]$ to \naturalFiltration{\timePoint} because the transformation from \randomProcess[\timePoint] to \transformedProcess[\timePoint] is injective.
    
    \newcommand*{\extraTime}{\mathOrText{r}}
    We now show that $\E{\firstHittingTime}[\randomProcess[0]] < \infty$ holds in order to apply Theorem~\ref{thm:optionalStopping}. Let $\extraTime > 0$, and let~$a$ be any value such that $\Pr{\randomProcess[0] \leq a} > 0$. We condition on the event $\{\randomProcess[0] \leq a\}$, and we consider a time point $\specialTimePoint = (a + \extraTime)/\delta$ and want to bound the probability that \randomProcess[\specialTimePoint] has not reached~$0$ yet, that is, the event $\{\randomProcess[\specialTimePoint] > 0\}$. We rewrite this event as $\{\randomProcess[\specialTimePoint] - a > -a\}$, which is equivalent to $\{\transformedProcess[\specialTimePoint] - a > -a + \delta\specialTimePoint = r\}$, by definition of~\transformedProcess and~\specialTimePoint.
    
    Note that, for all $\timePoint < \firstHittingTime$, $|\transformedProcess[\timePoint] - \transformedProcess[\timePoint + 1]| < c + \delta + 1$, as we assume that $\randomProcess[\timePoint] \leq c$. Thus, the differences of~\transformedProcess[\timePoint] are bounded and we can apply Theorem~\ref{thm:AzumaHoeffding} as follows, noting that $\transformedProcess[0] = \randomProcess[0] \leq a$, due to our condition on $\{\randomProcess[0] \leq a\}$:
    \[
        \Pr{\transformedProcess[\specialTimePoint] - a > r}[\randomProcess[0] \leq a] \leq \Pr{\transformedProcess[\specialTimePoint] - \transformedProcess[0] \geq r}[\randomProcess[0] \leq a] \leq \eulerE^{-\frac{r^2}{2\specialTimePoint(c + \delta + 1)^2}}\ .
    \]
    If we choose $\extraTime \geq a$, we get $\specialTimePoint \leq 2\extraTime/\delta$ and, thus,
    \[
        \Pr{\transformedProcess[\specialTimePoint] - \transformedProcess[0] > r}[\randomProcess[0] \leq a] \leq \eulerE^{-\frac{r\delta}{4(c + \delta + 1)^2}}\ .
    \]
    This means that the probability that \randomProcess[\specialTimePoint] has not reached~$0$ goes exponentially fast toward~$0$ as $\specialTimePoint$ (and, hence, \extraTime) goes toward $\infty$. Thus, the expected value of~\firstHittingTime is finite.

    Now we can use case~(\ref{item:optionalStoppingSupermartingale}) of Theorem~\ref{thm:optionalStopping} in order to get $\E{\transformedProcess[\firstHittingTime]}[\randomProcess[0]] \leq \E{\transformedProcess[0]}[\randomProcess[0]]$. In particular, noting that $\randomProcess[\firstHittingTime] = 0$ by definition,
    \begin{align*}
        \randomProcess[0] &= \E{\randomProcess[0]}[\randomProcess[0]] = \E{\transformedProcess[0]}[\randomProcess[0]]\\
        &\geq \E{\transformedProcess[\firstHittingTime]}[\randomProcess[0]] = \E{\randomProcess[\firstHittingTime] + \delta\firstHittingTime}[\randomProcess[0]] = \E{\randomProcess[\firstHittingTime]}[\randomProcess[0]] + \delta\E{\firstHittingTime}[\randomProcess[0]] = \delta\E{\firstHittingTime}[\randomProcess[0]]\ .
    \end{align*}
    Thus, we get the desired bound by dividing by $\delta$.\qed
\end{proof}

Note that the arguments in this proof only need the property of bounded differences in order to apply Theorem~\ref{thm:AzumaHoeffding}. Thus, we can relax the condition of a bounded state space into bounded step size, which can be seen in the following theorem.

\begin{theorem}[\textrm{Upper Additive Drift, Bounded Step Size}]
    \label{thm:addDriftBoundedSteps}
    Let $(\randomProcess[\timePoint])_{\timePoint \in \Na}$ be a random process over~\Re, and let $\firstHittingTime = \inf\{\timePoint \mid \randomProcess[\timePoint] \leq 0\}$. Furthermore, suppose that,
    \begin{enumerate}[label=(\alph*)]
        \item for all $\timePoint \leq \firstHittingTime$, it holds that $\randomProcess[\timePoint] \geq 0$, that
        
        \item there is some value $\delta > 0$ such that, for all $\timePoint < \firstHittingTime$, it holds that $\randomProcess[\timePoint] - \E{\randomProcess[\timePoint + 1]}[\naturalFiltration{\timePoint}] \geq \delta$, and that
        
        \item there is some value $c \geq 0$ such that, for all $\timePoint < \firstHittingTime$, it holds that $|\randomProcess[\timePoint + 1] - \randomProcess[\timePoint]| \leq c$.
    \end{enumerate}
    Then
    \[
        \E{\firstHittingTime}[\randomProcess[0]] \leq \frac{\randomProcess[0]}{\delta}\ .
    \]

\end{theorem}

Although the proof of Theorem~\ref{thm:addDriftBounded} can be used for Theorem~\ref{thm:addDriftBoundedSteps} as well, we provide a different proof strategy in the appendix, which we then generalize for our next theorem. This alternative strategy defines a process similar to~\randomProcess[\timePoint] that behaves like~\randomProcess[\timePoint] in the limit.

The proof of Theorem~\ref{thm:addDriftBoundedSteps} makes use of Theorem~\ref{thm:addDriftBounded} by artificially bounding the search space for a time that is sufficient in order to bound the expected first-hitting time. This approach can be used in order to let the restriction of the bounded step size fall entirely. Since we cannot make many assumptions about the process in this case anymore, we rely on Markov's inequality in order to show that our process will not leave, with sufficiently high probability, an interval large enough to properly bound the expected first-hitting time.

\begin{theorem}[\textrm{Upper Additive Drift, Unbounded}]
    \label{thm:addDriftUnbounded}
    Let $(\randomProcess[\timePoint])_{\timePoint \in \Na}$ be a random process over~\Re, and let $\firstHittingTime = \inf\{\timePoint \mid \randomProcess[\timePoint] \leq 0\}$. Furthermore, suppose that,
    \begin{enumerate}[label=(\alph*)]
        \item for all $\timePoint \leq \firstHittingTime$, it holds that $\randomProcess[\timePoint] \geq 0$, and that
        
        \item there is some value $\delta > 0$ such that, for all $\timePoint < \firstHittingTime$, it holds that $\randomProcess[\timePoint] - \E{\randomProcess[\timePoint + 1]}[\naturalFiltration{\timePoint}] \geq \delta$.
    \end{enumerate}
    Then
    \[
        \E{\firstHittingTime}[\randomProcess[0]] \leq \frac{\randomProcess[0]}{\delta}\ .
    \]
\end{theorem}

As we already mentioned before, note that the condition of the process not being negative is important in order to get correct results. The following example highlights this fact.

\begin{example}
    \label{ex:additiveDriftUpperBound}
    Let $n > 1$, and let $(\randomProcess[\timePoint])_{\timePoint \in \Na}$ be a random process with $\randomProcess[0] = 1$ and, for all $t \in \Na$, $\randomProcess[\timePoint + 1] = \randomProcess[\timePoint]$ with probability $1 - 1/n$, and $\randomProcess[\timePoint + 1] = - n + 1$ otherwise. Let~\firstHittingTime denote the first point in time~\timePoint such that the event $\randomProcess[\timePoint] \leq 0$ occurs. We have, for all $\timePoint < \firstHittingTime$, that $\randomProcess[\timePoint] - \E{\randomProcess[\timePoint + 1]}[\naturalFiltration{\timePoint}] = 1$ and, thus, $\E{\firstHittingTime}[\randomProcess[0]] \leq 1$ if we could apply any of the additive drift theorems. However, since~\firstHittingTime follows a geometric distribution with success probability $1/n$, we have $\E{\firstHittingTime}[\randomProcess[0]] = n$.
\end{example}

\subsection{Lower Bound}
\label{subsec:additiveDriftLowerBound}

In this section, we provide a lower bound for the expected first-hitting time under additive drift. In order to do so, we need an upper bound for the drift. Since we now lower-bound the first-hitting time, a large upper bound of the drift makes the result bad. Thus, we can allow the process to take negative values, as these could only increase the drift's upper bound. However, we need to have some restriction on the step size in order to make sure not to move away from the target. Again, we provide an example (Example~\ref{ex:additiveDriftLowerBound}) showing this necessity at the end of this section.

\begin{theorem}[\textrm{Lower Additive Drift, Expected Bounded Step Size}]
    \label{thm:addDriftLowerBoundExpectedBoundedSteps}
    Let $(\randomProcess[\timePoint])_{\timePoint \in \Na}$ be a random process over $\Re$, and let $\firstHittingTime = \inf\{\timePoint \mid \randomProcess[\timePoint] \leq 0\}$. Furthermore, suppose that
    \begin{enumerate}[label=(\alph*)]
        \item there is some value $\delta > 0$ such that, for all $\timePoint < \firstHittingTime$, it holds that $\randomProcess[\timePoint] - \E{\randomProcess[\timePoint + 1]}[\naturalFiltration{\timePoint}] \leq \delta$, and that
        
        \item there is some value $c \geq 0$ such that, for all $\timePoint < \firstHittingTime$, it holds that $\E{|\randomProcess[\timePoint + 1] - \randomProcess[\timePoint]|}[\naturalFiltration{\timePoint}][\big] \leq c$.
    \end{enumerate}
    Then
    \[
        \E{\firstHittingTime}[\randomProcess[0]] \geq \frac{\randomProcess[0]}{\delta}\ .
    \]
    
\end{theorem}
\begin{proof}
    We make a case distinction with respect to $\E{\firstHittingTime}[\randomProcess[0]]$ being finite. If $\E{\firstHittingTime}[\randomProcess[0]]$ is infinite, then the theorem trivially holds. Thus, we now assume that $\E{\firstHittingTime}[\randomProcess[0]] < \infty$.
    
    Similar to the proof of Theorem~\ref{thm:addDriftBounded}, we define, for all $\timePoint < \firstHittingTime$, $\transformedProcess[\timePoint] = \randomProcess[\timePoint] + \delta\timePoint$, which is a submartingale, since
    \begin{align*}
        \transformedProcess[\timePoint] - \E{\transformedProcess[\timePoint + 1]}[\transformedProcess[0], \ldots, \transformedProcess[\timePoint]] &= \randomProcess[\timePoint] - \delta\timePoint - \E{\randomProcess[\timePoint + 1] - \delta(\timePoint + 1)}[\naturalFiltration{\timePoint}]\\
        &= \randomProcess[\timePoint] - \E{\randomProcess[\timePoint + 1]}[\naturalFiltration{\timePoint}] - \delta \leq 0\ ,
    \end{align*}
    as we assume that $\randomProcess[\timePoint] - \E{\randomProcess[\timePoint + 1]}[\naturalFiltration{\timePoint}] \leq \delta$ for all $\timePoint < \firstHittingTime$ and because, again, the transformation of~\randomProcess[\timePoint] to~\transformedProcess[\timePoint] is injective.
    
    Since we now assume that both $\E{\firstHittingTime}[\randomProcess[0]] < \infty$ and, further, that $\E{|\randomProcess[\timePoint + 1] - \randomProcess[\timePoint]|}[\naturalFiltration{\timePoint}][\big] \leq c$ for all $\timePoint < \firstHittingTime$, we can directly apply case~(\ref{item:optionalStoppingSubmartingale}) of Theorem~\ref{thm:optionalStopping} and get that $\E{\transformedProcess[\firstHittingTime]}[\randomProcess[0]] \geq \E{\transformedProcess[0]}[\randomProcess[0]]$. This yields, noting that $\randomProcess[\firstHittingTime] \leq 0$,
    \begin{align*}
        \randomProcess[0] &= \E{\randomProcess[0]}[\randomProcess[0]] = \E{\transformedProcess[0]}[\randomProcess[0]]\\
        &\leq \E{\transformedProcess[\firstHittingTime]}[\randomProcess[0]] = \E{\randomProcess[\firstHittingTime] + \delta\firstHittingTime}[\randomProcess[0]] = \E{\randomProcess[\firstHittingTime]}[\randomProcess[0]] + \delta\E{\firstHittingTime}[\randomProcess[0]] \leq \delta\E{\firstHittingTime}[\randomProcess[0]]\ .
    \end{align*}
    Thus, we get the desired bound by dividing by $\delta$.\qed
\end{proof}

Note that the step size has to be bounded in some way for a lower bound, as the following example shows.

\begin{example}
    \label{ex:additiveDriftLowerBound}
    Let $\delta \in (0,1)$, and let $(\randomProcess[\timePoint])_{\timePoint \in \Na}$ be a random process with $\randomProcess[0] = 2$ and, for all $t \in \Na$, $\randomProcess[\timePoint + 1] = 0$ with probability $1/2$ and $\randomProcess[\timePoint + 1] = 2\randomProcess[\timePoint] - 2\delta$ otherwise. Further, let~\firstHittingTime  denote the first point in time~\timePoint such that $\randomProcess[\timePoint] = 0$. Then~\firstHittingTime follows a geometric distribution with success probability $1/2$, which yields $\E{\firstHittingTime} = 2$. However, we have that $\randomProcess[\timePoint] - \E{\randomProcess[\timePoint + 1]}[\naturalFiltration{\timePoint}] = \delta$. If Theorem~\ref{thm:addDriftLowerBoundExpectedBoundedSteps} could be applied to this process (by neglecting the condition of the bounded step size), the theorem would yield that $\E{\firstHittingTime} \geq 2/\delta$, which is not true.
\end{example}

    \section{Variable Drift}
    \label{sec:variableDrift}
    
In contrast to additive drift, \emph{variable} drift means that the drift can depend on the current state of the process (while still being bounded independently of the time). Interestingly, these more flexible drift theorems can be derived by using additive drift. Intuitively, the reasoning behind this approach is to scale the search space such that the information relevant to the process's history cancels out.

It is important to note that variable drift theorems are commonly phrased such that the first-hitting time~\firstHittingTime denotes the first point in time such that the random process drops \emph{below} a certain value (our target)~-- it is not enough to hit that value. However, this restriction is not always necessary. Thus, we also consider the setting from Section~\ref{sec:additiveDrift}, where~\firstHittingTime denotes the first point in time such that we \emph{hit} our target. In this section, our target is no longer~$0$ but a value~\xmin.

In all of our theorems in this section, we make use of a set~$D$. This set contains (at least) all possible values that our process can take while not having reached the target yet. It is a formal necessity in order to calculate the bound of the first-hitting time (via an integral). However, when applying the theorem, it is usually sufficient to choose $D = \Re$ or $D = \Re_{\geq 0}$.

\medskip
The first variable drift theorem was proven by Johannsen~\cite{Johannsen10} and, independently in a different version, by Mitavskiy et al.~\cite{DBLP:journals/ijicc/MitavskiyRC09}. It was later refined by Rowe and Sudholt~\cite{RoweS14}. In all of these versions, bounded search spaces were used. Due to Theorem~\ref{thm:addDriftUnbounded}, we can drop this restriction.

\subsubsection{Going below the target.}
\renewcommand*{\d}{\mathrm{d}}

The following version of the theorem assumes that the process has to drop below the target, denoted by~\xmin. We provide the other version afterward.

\begin{theorem}[Upper Variable Drift, Unbounded, Below Target]
    \label{thm:varDriftUnbounded}
    Let $(\randomProcess[\timePoint])_{\timePoint \in \Na}$ be a random process over~\Re, $\xmin > 0$, and let $\firstHittingTime = \inf\{\timePoint \mid \randomProcess[\timePoint] < \xmin\}$. Additionally, let~$D$ denote the smallest real interval that contains at least all values $x \geq \xmin$ that, for all $\timePoint \leq \firstHittingTime$, any~\randomProcess[\timePoint] can take. Furthermore, suppose that
    \begin{enumerate}[label=(\alph*)]
        
        \item $\randomProcess[0] \geq \xmin$ and, for all $\timePoint \leq \firstHittingTime$, it holds that $\randomProcess[\timePoint] \geq 0$ and that
        
        \item there is a monotonically increasing function $h\colon D \to \Re^+$ such that, for all $\timePoint < \firstHittingTime$, we have $\randomProcess[\timePoint] - \E{\randomProcess[\timePoint + 1]}[\naturalFiltration{\timePoint}] \geq h(\randomProcess[\timePoint])$.
    \end{enumerate}
    Then
    \[
        \E{\firstHittingTime}[\randomProcess[0]] \leq \frac{\xmin}{h(\xmin)} + \int_{\xmin}^{\randomProcess[0]} \frac{1}{h(z)} \d z\ .
    \]
\end{theorem}

\subsubsection{Hitting the target.}

As mentioned before, it is not always necessary to drop below the target. For the additive drift, for example, we are interested in the first time reaching the target. Interestingly, the proof for the following theorem is straightforward, as it is almost the same as the proof of Theorem~\ref{thm:varDriftUnbounded}. Intuitively, the waiting time for getting below the target, once it is reached, is eliminated from the expected first-hitting time. However, it is important to note that it is now not allowed to get below the target.

\begin{theorem}[Upper Variable Drift, Unbounded, Hitting Target]
    \label{thm:varDriftUnboundedNoGap}
    Let $(\randomProcess[\timePoint])_{\timePoint \in \Na}$ be a random process over~\Re, $\xmin \geq 0$, and let $\firstHittingTime = \inf\{\timePoint \mid \randomProcess[\timePoint] \leq \xmin\}$. Additionally, let~$D$ denote the smallest real interval that contains at least all values $x \geq \xmin$ that, for all $\timePoint \leq \firstHittingTime$, any~\randomProcess[\timePoint] can take. Furthermore, suppose that,
    \begin{enumerate}[label=(\alph*)]
        \item for all $\timePoint \leq \firstHittingTime$, it holds that $\randomProcess[\timePoint] \geq \xmin$ and that
        
        \item there is a monotonically increasing function $h\colon D \to \Re^+$ such that, for all $\timePoint < \firstHittingTime$, we have $\randomProcess[\timePoint] - \E{\randomProcess[\timePoint + 1]}[\naturalFiltration{\timePoint}] \geq h(\randomProcess[\timePoint])$.
    \end{enumerate}
    Then
    \[
        \E{\firstHittingTime}[\randomProcess[0]] \leq \int_{\xmin}^{\randomProcess[0]} \frac{1}{h(z)} \d z\ .
    \]
\end{theorem}

    \section{Multiplicative Drift}
    \label{sec:multiplicativeDrift}
    
A special case of variable drift is \emph{multiplicative} drift, where the drift can be bounded by a multiple of the most recent value in the history of the process. As before, we provide upper bounds in the two versions of either dropping below the target or hitting it. In this setting, it can be intuitively argued why the version of dropping below the target is useful: consider a sequence of nonnegative numbers that halves its current value each time step. This process will never reach~$0$ within finite time. However, it drops below any value greater than~$0$.

Both upper bounds we state are simple applications of the corresponding variable drift theorems from Section~\ref{sec:variableDrift}.

\subsubsection{Going below the target.}

Corollary~\ref{cor:MultiDriftUnbounded} has first been stated by Doerr et al.~\cite{DBLP:conf/gecco/DoerrJW10} using finite state spaces. However, a closer look at the proof shows that this restriction is not necessary.

\begin{corollary}[Upper Multiplicative Drift, Unbounded, Below Target]
    \label{cor:MultiDriftUnbounded}
    Let $(\randomProcess[\timePoint])_{\timePoint \in \Na}$ be a random process over~\Re, $\xmin > 0$, and let $\firstHittingTime = \inf\{\timePoint \mid \randomProcess[\timePoint] < \xmin\}$. Furthermore, suppose that
    \begin{enumerate}[label=(\alph*)]
        \item $\randomProcess[0] \geq \xmin$ and, for all $\timePoint \leq \firstHittingTime$, it holds that $\randomProcess[\timePoint] \geq 0$, and that
        
        \item there is some value $\delta > 0$ such that, for all $\timePoint < \firstHittingTime$, it holds that $\randomProcess[\timePoint] - \E{\randomProcess[\timePoint + 1]}[\naturalFiltration{\timePoint}][\big] \geq \delta\randomProcess[\timePoint]$.
    \end{enumerate}
    Then
    \[
        \E{\firstHittingTime}[\randomProcess[0]] \leq \frac{1 + \ln\left(\frac{\randomProcess[0]}{\xmin}\right)}{\delta}\ .
    \]
\end{corollary}

\subsubsection{Hitting the target.}

By applying Theorem~\ref{thm:varDriftUnboundedNoGap} instead of Theorem~\ref{thm:varDriftUnbounded}, we get the following theorem. As in the case of Theorem~\ref{thm:varDriftUnboundedNoGap}, the process now has to be lower-bounded by~\xmin.

\begin{corollary}[Upper Multiplicative Drift, Unbounded, Hitting Target]
    \label{cor:MultiDriftUnboundedNoGap}
    Let $(\randomProcess[\timePoint])_{\timePoint \in \Na}$ be a random process over~\Re, $\xmin > 0$, and let $\firstHittingTime = \inf\{\timePoint \mid \randomProcess[\timePoint] \leq \xmin\}$. Furthermore, suppose that,
    \begin{enumerate}[label=(\alph*)]
        \item for all $\timePoint \leq \firstHittingTime$, it holds that $\randomProcess[\timePoint] \geq \xmin$, and that
        
        \item there is some value $\delta > 0$ such that, for all $\timePoint < \firstHittingTime$, it holds that $\randomProcess[\timePoint] - \E{\randomProcess[\timePoint + 1]}[\naturalFiltration{\timePoint}][\big] \geq \delta\randomProcess[\timePoint]$.
    \end{enumerate}
    Then
    \[
        \E{\firstHittingTime}[\randomProcess[0]] \leq \frac{\ln\left(\frac{\randomProcess[0]}{\xmin}\right)}{\delta}\ .
    \]
\end{corollary}

Again, we provide an example that shows that the bounds above are as tight as possible, up to constant factors, for the range of processes we consider.
The example describes a process that decreases deterministically, that is, it has a variance of~$0$.

\begin{example}
    \label{ex:multiDriftUpperBound}
    Let $\delta \in (0, 1)$ be a value bounded away from~$1$. Consider the process $(\randomProcess[\timePoint])_{\timePoint \in \Na}$, with $\randomProcess[0] > 1$, that decreases each step deterministically such that $\randomProcess[\timePoint + 1] = (1 - \delta)\randomProcess[\timePoint]$ holds. Let~\firstHittingTime denote the first point in time such that the process drops below~$1$. Thus, we get $\firstHittingTime = \Theta(-\log_{(1 - \delta)}\randomProcess[0]) = \Theta\big(-\ln(\randomProcess[0])/\ln(1 - \delta)\big) = \Theta\big(\ln(\randomProcess[0])/\delta\big)$, where the last equation makes use of the Taylor expansion of $\ln(1 - \delta) = \Theta(-\delta)$, as $1 - \delta$ does not converge to~$0$, by assumption.
\end{example}

    \bibliographystyle{splncs03}
    \bibliography{ElegantDrift}
    
    \clearpage
    \appendix
\section{Appendix}

\subsection{Proof of Theorem~\ref{thm:addDriftBoundedSteps}}

Before we prove Theorem~\ref{thm:addDriftBoundedSteps}, we state and prove the following lemma, which we are then going to use in the proof of Theorem~\ref{thm:addDriftBoundedSteps}.

\begin{lemma}
    \label{lem:convergenceLemma}
    Let $X$ be a random variable over $\Na$ and $(X_n)_{n \in \Na}$ a sequence of random variables over $\Na$. If, for all $x \in \Na$, it holds that $\Pr{X=x} \leq \lim_{n \rightarrow \infty} \Pr{X_n=x}$, then $\E{X} \leq \lim_{n \rightarrow \infty} \E{X_n}$.
\end{lemma}
\begin{proof}
    First we show, for all $x$, that the condition $\Pr{X=x} \leq \lim_{n \rightarrow \infty} \Pr{X_n=x}$ implies $\Pr{X=x} = \lim_{n \rightarrow \infty} \Pr{X_n=x}$. Assume, by way of contradiction, that there is an $\varepsilon > 0$ and an $x' \in \Na$ such that $\lim_{n \to \infty} \Pr{X_n = x'} = \Pr{X = x'} + \varepsilon$. Let $k \geq x'$ be such that $\sum_{x = 0}^{k} \Pr{X = x} > 1 - \varepsilon/2$. Since $X_n$ converges to~$X$, choose an $n_0 \in \Na$ and a $\delta \in \Re$ with $0 \leq \delta < \varepsilon/\big(2(k + 1)\big)$ such that, for all $x \in \{0, \ldots, k\}$, $\Pr{X_{n_0} = x} \geq \Pr{X = x} - \delta$ if $x \neq x'$, and $\Pr{X_{n_0} = x'} \geq \Pr{X = x'} + \varepsilon - \delta$ otherwise. Then we have
    \begin{align*}
        \sum_{x = 0}^{k} \Pr{X_{n_0} = x} > 1 - \frac{\varepsilon}{2} - (k + 1)\delta + \varepsilon = 1 - (k + 1)\delta + \frac{\varepsilon}{2} > 1\ ,
    \end{align*}
    since $\varepsilon/2 > (k + 1)\delta$. This contradicts that $X_{n_0}$ follows a probability distribution. Thus, for all $x \in \Na$, $\Pr{X=x} = \lim_{n \rightarrow \infty} \Pr{X_n=x}$.
    
    Now, let $a = \lim_{n \rightarrow \infty} \E{X_n}$.
    Suppose, by way of contradiction, $\E{X} > a$. Then there is a $k$ and an $\varepsilon > 0$ such that
    \[
    \sum_{i=1}^k i \cdot \Pr{X = i} \geq a + \varepsilon\ .
    \]
    Let $n$ be large enough such that, for all $i$ with $1 \leq i \leq k$, $\big|\Pr{X_n = i} - \Pr{X = i}\big| \leq \varepsilon/(i \cdot 2^i)$. We now have
    \begin{align*}
        \E{X_n} &= \sum_{i=1}^\infty i \cdot \Pr{X_n = i} \geq \sum_{i=1}^k i \cdot \Pr{X_n = i}\\
        &\geq \left(\sum_{i=1}^k i \cdot \left(\Pr{X = i} - \frac{\varepsilon}{i \cdot 2^i}\right)\right)\\
        &=    \left(\sum_{i=1}^k i \cdot \Pr{X = i}\right) - \sum_{i=1}^k \frac{\varepsilon}{2^i}\\
        &> \left(\sum_{i=1}^k i \cdot \Pr{X = i}\right) - \varepsilon \geq a\ ,
    \end{align*}
    a contradiction.\qed
\end{proof}

\begin{proof}[of Theorem~\ref{thm:addDriftBoundedSteps}]
    Let~$a$ be any value such that $\Pr{\randomProcess[0] \leq a} > 0$. For any $z > a$ and any $\timePoint \in \Na$, let $A_{z, \timePoint}$ be the event that, for all $\specialTimePoint \leq \timePoint$, $\{\randomProcess[\specialTimePoint] \leq z\}$. Consider the process $(\randomProcess[\timePoint][z])_{\timePoint \in \Na}$ with $\randomProcess[\timePoint][z] = \randomProcess[\timePoint]$ if $A_{z, \timePoint}$ is true and $\randomProcess[\timePoint][z] = 0$ otherwise. Let $\firstHittingTime[z]$ be the first-hitting time of~$0$ of this modified process. Then the process is a random process on $[0, z]$ with a drift of at least~$\delta$ toward~$0$, so Theorem~\ref{thm:addDriftBounded} yields $\E{\firstHittingTime[z]}[\randomProcess[0][z]][\big] \leq \randomProcess[0][z]/\delta$. Note that $\randomProcess[0][z] \leq \randomProcess[0]$, since the values of $\randomProcess[0][z]$ and $\randomProcess[0]$ either coincide (if $A_{z, \timePoint}$ is true) or (if $A_{z, \timePoint}$ is false) $\randomProcess[0][z] = 0 \leq \randomProcess[0]$, since $\randomProcess[\timePoint]$ is non-negative for all $\timePoint \leq \firstHittingTime$. Further, the expected value of $\firstHittingTime[z]$ is the same when conditioning on $\randomProcess[0]$ or $\randomProcess[0][z]$, since $A_{z, \timePoint}$ is defined with respect to $\randomProcess[\timePoint]$. Thus, $\E{\firstHittingTime[z]}[\randomProcess[0]] = \E{\firstHittingTime[z]}[\randomProcess[0][z]][\big] \leq \randomProcess[0]/\delta$.
    
    Since the step width of~\randomProcess is bounded by $c$, we have, for all $z \geq a + ck$, that $\Pr{\firstHittingTime[z] = k}[\randomProcess[0] \leq a][\big] = \Pr{\firstHittingTime = k}[\randomProcess[0] \leq a]$: during~$k$ steps, both processes cannot exceed~$z$. Thus, they are the same, by construction. Since~$a$ is arbitrary, it follows that $\Pr{\firstHittingTime[z] = k}[][\big] = \Pr{\firstHittingTime = k}$. Using Lemma~\ref{lem:convergenceLemma}, we now see $\E{T}[\randomProcess[0]] \leq \lim_{z \to \infty} \E{\firstHittingTime[z]}[\randomProcess[0]][\big] \leq \randomProcess[0]/\delta$.\qed
\end{proof}

\subsection{Proof of Theorem~\ref{thm:addDriftUnbounded}}

\begin{proof}[of Theorem~\ref{thm:addDriftUnbounded}]
    We use the same approach as in the proof of Theorem~\ref{thm:addDriftBoundedSteps}. Thus, we use the same notation of~$a$, $A_{z, \timePoint}$, and $\randomProcess[\timePoint][z]$. However, this time, we cannot bound the step size of~\randomProcess deterministically. Thus, we do so probabilistically. In the following, we condition on the event $\{\randomProcess[0] \leq a\}$ without denoting this explicitly.
    
    By the law of total probability, we have, for all $k \in \Na$, all $z > a$, and all $\timePoint \geq k$,
    \begin{align*}
        \Pr{\firstHittingTime = k} &= \Pr{\firstHittingTime = k}[A_{z, \timePoint}]\Pr{A_{z, \timePoint}} + \Pr{\firstHittingTime = k}[\overline{A_{z, \timePoint}}]\Pr{\overline{A_{z, \timePoint}}}\\
        &\leq \Pr{\{\firstHittingTime = k\} \cap A_{z, \timePoint}} + \Pr{\overline{A_{z, \timePoint}}}\ .
    \end{align*}
    Note that $\Pr{\firstHittingTime = k}[A_{z, \timePoint}] = \Pr{\firstHittingTime[z] = k}[A_{z, \timePoint}]$, as $\randomProcess[][z] = \randomProcess$, due to the condition~$A_{z, \timePoint}$. Thus,
    \begin{align*}
        \Pr{\firstHittingTime = k} &\leq \Pr{\{\firstHittingTime[z]  = k\} \cap A_{z, \timePoint}} + \Pr{\overline{A_{z, \timePoint}}}\\
        &\leq \Pr{\firstHittingTime[z]  = k}[][\big] + \Pr{\overline{A_{z, \timePoint}}}\ .
    \end{align*}
    We now show that $\Pr{\overline{A_{z, \timePoint}}}[][\big]$ goes to~$0$ as~$z$ goes to infinity. This will establish $\Pr{\firstHittingTime = k} \leq \lim_{z \to \infty} \Pr{\firstHittingTime[z] = k}[][\big]$.
    
    Due to Markov's inequality, we get, for any $c > 0$ and all $\timePoint < \firstHittingTime$, that $\Pr{\randomProcess[\timePoint + 1] > c\randomProcess[\timePoint]}[\naturalFiltration{\timePoint}] \leq 1/c$, as $\E{\randomProcess[\timePoint + 1]}[\naturalFiltration{\timePoint}] \leq \randomProcess[\timePoint] - \delta < \randomProcess[\timePoint]$. Thus, inductively, we get $\Pr{\randomProcess[\timePoint + 1] > c^{\timePoint + 1}\randomProcess[0]}[\naturalFiltration{\timePoint}] \leq (\timePoint + 1)/c$ via a union bound by pessimistically assuming that $\{\randomProcess[\timePoint + 1] > c^{\timePoint + 1}\randomProcess[0]\}$ already holds if, for any $\specialTimePoint \leq \timePoint$, $\{\randomProcess[\specialTimePoint + 1] > c\randomProcess[\specialTimePoint]\}$. By defining $z = c^{\timePoint + 1}a$, we get that $\Pr{\randomProcess[\timePoint + 1] > z}[\naturalFiltration{\timePoint}] \leq (\timePoint + 1)/\sqrt[\timePoint + 1]{z/a}$.
    
    In order for $\overline{A_{z, \timePoint}}$ to occur, it is sufficient that there is a $\specialTimePoint \leq \timePoint$ such that the event $\{\randomProcess[\specialTimePoint] > z\}$ occurs. Hence, via another union bound over all of these possibilities, we get
    \[
    \Pr{\overline{A_{z, \timePoint}}} \leq \sum_{\specialTimePoint = 0}^{\timePoint} \frac{(\specialTimePoint + 1)\sqrt[\specialTimePoint + 1]{a}}{\sqrt[\specialTimePoint + 1]{z}} \leq \frac{(\timePoint + 1)^2 a}{\sqrt[\timePoint + 1]{z}}\ ,
    \]
    which goes to~$0$ as~$z$ approaches infinity, since~\timePoint is fixed.
    
    Overall, we get that, for all $k \in \Na$, $\lim_{z \to \infty} \Pr{\firstHittingTime[z] = k}[][\big] \geq \Pr{\firstHittingTime = k}$. By applying Lemma~\ref{lem:convergenceLemma}, we now see that $\E{T}[\randomProcess[0]] \leq \lim_{z \to \infty} \E{T[z]}[\randomProcess[0]][\big] \leq \randomProcess[0]/\delta$, where the last inequality follows from an application of Theorem~\ref{thm:addDriftBounded}.\qed
\end{proof}

\subsection{Proof of Theorem~\ref{thm:varDriftUnbounded}}

\begin{proof}[of Theorem~\ref{thm:varDriftUnbounded}]
    The proof follows the one given by Rowe and Sudholt~\cite{RoweS14} very closely. We define a function $g\colon D \cup [0, \xmin] \to \Re_{\geq 0}$ as follows:
    \[
    g(x) =
    \begin{cases}
    0 &\textrm{if $x < \xmin$,}\\
    \frac{\xmin}{h(\xmin)} + \int_{\xmin}^{x} \frac{1}{h(z)} \d z &\textrm{else.}
    \end{cases}
    \]
    Note that~$g$ is well-defined, since $1/h$ is monotonically decreasing and every monotone function is integrable over all compact intervals of its domain. Further, $g(\randomProcess[\timePoint]) = 0$ holds if and only if $\randomProcess[\timePoint] < \xmin$. Thus, both processes have the same first-hitting time.
    
    Assume that $x \geq y \geq \xmin$. We get
    \[
    g(x) - g(y) = \int_{y}^{x} \frac{1}{h(z)} \d z \geq \frac{x - y}{h(x)}\ ,
    \]
    since~$h$ is monotonically increasing. Assuming $y \geq x \geq \xmin$, we get, similar to before,
    \[
    g(x) - g(y) = -\int_{x}^{y} \frac{1}{h(z)} \d z \geq -\frac{y - x}{h(x)} = \frac{x - y}{h(x)}\ .
    \]
    Thus, we can write, for $x \geq \xmin$ and $y \geq \xmin$,
    \[
    g(x) - g(y) \geq \frac{x - y}{h(x)}\ .
    \]
    
    Further, for $x \geq \xmin > y \geq 0$, we get
    \begin{align*}
        g(x) - g(y) &= \frac{\xmin}{h(\xmin)} + \int_{\xmin}^{x} \frac{1}{h(z)} \d z \geq \frac{\xmin}{h(x)} + \frac{x - \xmin}{h(x)}\\
        &= \frac{x}{h(x)} \geq \frac{x - y}{h(x)}\ .
    \end{align*}
    
    Overall, for $x \geq \xmin$ (including $\randomProcess[0] \geq \xmin$) and $y \in \Re_{\geq 0}$, we can estimate
    \[
    g(x) - g(y) \geq \frac{x - y}{h(x)}\ .
    \]
    
    We use this to determine the drift of the process $g(\randomProcess[\timePoint])$ as follows:
    \begin{align*}
        g(\randomProcess[\timePoint]) - \E{g(\randomProcess[\timePoint + 1])}[\naturalFiltration{\timePoint}] &= \E{g(\randomProcess[\timePoint]) - g(\randomProcess[\timePoint + 1])}[\naturalFiltration{\timePoint}]\\
        &\geq \frac{\E{\randomProcess[\timePoint] - \randomProcess[\timePoint + 1]}[\naturalFiltration{\timePoint}]}{h(\randomProcess[\timePoint])}\\
        &\geq 1\ ,
    \end{align*}
    where we used the condition on the drift of $\randomProcess[\timePoint]$.
    
    An application of Theorem~\ref{thm:addDriftUnbounded} completes the proof.\qed
\end{proof}

\subsection{Proof of Theorem~\ref{thm:varDriftUnboundedNoGap}}

\begin{proof}[of Theorem~\ref{thm:varDriftUnboundedNoGap}]
    This proof is almost identical to the proof of Theorem~\ref{thm:varDriftUnbounded}. The difference is that we define our potential function $g\colon D \to \Re_{\geq 0}$ as follows:
    \[
    g(x) =
    \begin{cases}
    0 &\textrm{if $x \leq \xmin$,}\\
    \int_{\xmin}^{x} \frac{1}{h(z)} \d z &\textrm{else.}
    \end{cases}
    \]
    As for $g(x) - g(y)$, the case $x \geq \xmin > y$ does not exist anymore, since we cannot get below~\xmin. Thus, the potential difference is the same in all cases, and nothing changes in the rest of the proof.\qed
\end{proof}

\subsection{Proof of Corollaries~\ref{cor:MultiDriftUnbounded} and~\ref{cor:MultiDriftUnboundedNoGap}}

\begin{proof}[of Corollary~\ref{cor:MultiDriftUnbounded}]
    We define a function $h\colon [\xmin, \infty) \to \Re^+$ with $h(x) = \delta x$. Note that~$h$ is monotonically increasing and that, by construction, for all $\timePoint < \firstHittingTime$, $\randomProcess[\timePoint] - \E{\randomProcess[\timePoint + 1]}[\naturalFiltration{\timePoint}] \geq h(\randomProcess[\timePoint])$. Thus, by applying Theorem~\ref{thm:varDriftUnbounded}, we get
    \[
    \E{\firstHittingTime}[\randomProcess[0]] \leq \frac{\xmin}{h(\xmin)} + \int_{\xmin}^{\randomProcess[0]} \frac{1}{h(z)} \d z = \frac{\xmin}{\delta\xmin} + \frac{\ln\left(\frac{\randomProcess[0]}{\xmin}\right)}{\delta}\ ,
    \]
    which completes the proof.\qed
\end{proof}

\begin{proof}[of Corollary~\ref{cor:MultiDriftUnboundedNoGap}]
    We define the same potential as in the proof of Corollary~\ref{cor:MultiDriftUnbounded} but apply Theorem~\ref{thm:varDriftUnboundedNoGap} instead.\qed
\end{proof}

\end{document}